\documentclass{amsart}

\usepackage{amsfonts}
\usepackage{amssymb}
\usepackage{amsmath}

\usepackage{xcolor} %For colored text

\usepackage[hidelinks]{hyperref} % for refernce links
\hypersetup{
    colorlinks=true,
    linkcolor=red,
    citecolor=blue,
    filecolor=red,
   urlcolor=red}

\newcommand{\bsigma}{\mathbf{\Sigma}}
\newcommand{\bpi}{\mathbf{\Pi}}
\newcommand{\bdelta}{\mathbf{\Delta}}
\newcommand{\bgamma}{\mathbf{\Gamma}}
\newcommand{\spaceofg}{\mathcal{G}}
\def\om{\alpha}

\newcommand{\m}[1]{\mathcal{#1}}

\newtheorem{theorem}{Theorem}[section]

\newtheorem{lemma}[theorem]{Lemma}

\newtheorem{proposition}{Proposition}

\begin{document}

\title{Descriptive complexity of subsets of the space of finitely generated groups}
\author{Mustafa G\"{o}khan Benl\.{I}}
\address{Department of Mathematics, Middle East Technical University, 06800, \c{C}ankaya, Ankara, Turkey\\
}
\email{benli@metu.edu.tr}

\author{Burak Kaya}
\address{Department of Mathematics, Middle East Technical University, 06800, \c{C}ankaya, Ankara, Turkey\\
}
\email{burakk@metu.edu.tr}


\begin{abstract} In this paper, we determine the descriptive complexity of subsets of the Polish space of marked groups defined by various group theoretic properties. In particular, using Grigorchuk groups, we establish that the sets of solvable groups, groups of exponential growth and groups with decidable word problem are $\bsigma^0_2$-complete and that the sets of periodic groups and groups of intermediate growth are $\bpi^0_2$-complete. We also provide bounds for the descriptive complexity of simplicity, amenability, residually finiteness, Hopficity and co-Hopficity. This paper is intended to serve as a compilation of results on this theme.
\end{abstract}

\maketitle

\tableofcontents

\section{Introduction}

Under appropriate coding and identification, one can form a compact Polish space of $n$-generated marked groups. This space was introduced in \cite{Grigorchuk84} with the aim of studying a Cantor set of groups with unusual properties related to growth and amenability and has been well-studied since then. (For example, see  \cite{Grigorchuk84}, \cite{Cham}, \cite{Gri05}, \cite{Champetier05} and \cite{CGP07}.) In addition to being an interesting mathematical object on its own, this space can also be used to show the existence of groups with certain properties, such as being amenable but not elementary amenable. (See \cite{Ste96} and \cite{WW2017} for examples of such arguments.)

A theme that has been explored, though not studied in its full extent, is the analysis of the interplay between group theoretic properties and topological properties of the sets that they define in this space.
 
In this paper, we aim to collect results and open questions on this theme and initiate a systematic study for future research. We also provide some important corollaries of a classical construction of Grigorchuk and its modifications, regarding solvability, periodicity, decidability and growth properties, which, as far as we know, have remained unnoticed. In particular, we prove the following.

\begin{theorem} Let $n > 1$ be an integer and $\spaceofg_n$ be the Polish space of $n$-generated marked groups. Then
\begin{itemize}
\item[a.] the set of solvable groups in $\spaceofg_n$ is $\bsigma^0_2$-complete.
\item[b.] the set of periodic groups in $\spaceofg_n$ is $\bpi^0_2$-complete. 
\item[c.] the set of groups with exponential growth in $\spaceofg_n$ is $\bsigma^0_2$-complete. 
\item[d.] the set of groups with intermediate growth in $\spaceofg_n$ is $\bpi^0_2$-complete. 
\item[e.] the set of groups with decidable word problem in $\spaceofg_n$ is $\bsigma^0_2$-complete. 
\end{itemize} 
\end{theorem}

This paper is organized as follows. In Section 2, we shall summarize the relevant definitions and results from descriptive set theory and recall the construction of the Polish space of $n$-generated marked groups, together with some convergence results in this space for expository purposes. In Section 3, we will first set up a framework to formalize the notion of ``having a property $P$" to prove some general facts regarding descriptive complexity, for example, to analyze the relationship between ``being $P$" and ``being virtually $P$." Then we will prove our main result and determine bounds for the descriptive complexity of various important group theoretic properties. In Section 4, we will briefly mention properties that potentially define non-Borel sets and conclude by listing some open questions that we think are of importance.

\section{Preliminaries}
\subsection{Borel hierarchy and complete sets}
In this subsection, we will cover some classical results from descriptive set theory that are necessary for the remainder of this paper. We refer the reader to \cite{Kechris95} for a general background.

Recall that a \textit{Polish space} is a completely metrizable separable topological space. For the rest of this subsection, let $X$ be an uncountable Polish space. The Borel $\sigma$-algebra of $X$, i.e., the $\sigma$-algebra generated by the open subsets of $X$, can be stratified as follows. By transfinite recursion, for every countable ordinal $1 \leq \alpha < \omega_1$, define the following collections of subsets of $X$:
\begin{itemize}
\item $\bsigma^0_1=\{U \subseteq X: U\ \text{is open}\}$,
\item $\bpi^0_{\alpha}=\{X-S: S \in \bsigma^0_{\alpha}\}$ and $\bdelta^0_{\alpha}=\bsigma^0_{\alpha} \cap \bpi^0_{\alpha}$ for every $1 \leq \alpha < \omega_1$, and
\item $\bsigma^0_{\alpha}=\{\bigcup A_n: A_n \in \bpi^0_{\gamma_n}, 1 \leq \gamma_n < \alpha, n \in \mathbb{N}\}$ for every $1 < \alpha < \omega_1$.
\end{itemize}
In the classical terminology, the sets in $\bdelta^0_1$, $\bsigma^0_1$, $\bpi^0_1$, $\bsigma^0_2$ and $\bpi^0_2$ are the clopen, open, closed, $F_{\sigma}$ and $G_{\delta}$ subsets of $X$ respectively. It is easily checked that the Borel $\sigma$-algebra of $X$ is exactly the union of these classes, which together constitute the \textit{Borel hierarchy} of $X$. The Borel hierarchy of $X$ can be pictured as follows, where every class in the diagram is a subset of the classes on right of it.
\begin{align*}
\ \ & \bsigma^0_1 & \ \ & \bsigma^0_2 & \ \ & \bsigma^0_3 & \ \ & \dots & \ \ & \bsigma^0_{\alpha} & \ \\
\bdelta^0_1 & \ \ & \bdelta^0_2 & \ \ & \bdelta^0_3 & \ \ & \dots & \ \ & \bdelta^0_{\alpha} & \ & \alpha<\omega_1\\
\ \ & \bpi^0_1 & \ \ & \bpi^0_2 & \ \ & \bpi^0_3 & \ \ & \dots & \ \ & \bpi^0_{\alpha} & \
\end{align*}
It is well-known that, for an uncountable Polish space $X$, the Borel hierarchy does not collapse, i.e., the containments in the diagram above are all proper. This hierarchy provides a notion of complexity for Borel sets. Intuitively speaking, where a Borel set resides in this hierarchy measures how difficult it is to ``define" this set.

Given a topological space $Y$ and subsets $A \subseteq X$ and $B \subseteq Y$, we say that $B$ is \textit{Wadge reducible} to $A$ (or, $A$ \textit{Wadge reduces} $B$), written $B \leq_W A$, if there exists a continuous map $f: Y \rightarrow X$ such that $f^{-1}[A]=B$.

Next will be introduced the notion of a complete set. Let $\bgamma^0_{\alpha}$ be one of the classes $\bsigma^0_{\alpha}$ and $\bpi^0_{\alpha}$. A set $A \subseteq X$ is said to be $\bgamma^0_{\alpha}$-\textit{complete} if $A$ is in $\bgamma^0_{\alpha}$ and for every zero-dimensional Polish space $Y$ and every $\bgamma^0_{\alpha}$-set $B \subseteq Y$ we have that $B \leq_W A$.

It is straightforward to check that if $A \subseteq X$ is $\bgamma^0_{\alpha}$-complete, then $X-A$ is $\overline{\bgamma^0_{\alpha}}$-complete, where $\overline{\bgamma^0_{\alpha}}$ denotes the complementary class of $\bgamma^0_{\alpha}$. Moreover, any set Wadge reducing a $\bgamma^0_{\alpha}$-complete set is $\bgamma^0_{\alpha}$-complete. The following characterization of complete sets is well-known.

\begin{theorem}\cite[22.10]{Kechris95} Let $X$ be a zero-dimensional Polish space and $A $ be a subset of $X$. Then $A$ is $\bsigma^0_{\alpha}$-complete (respectively, $\bpi^0_{\alpha}$-complete) if and only if $A$ is in $\bsigma^0_{\alpha}-\bpi^0_{\alpha}$ (respectively, in $\bpi^0_{\alpha}-\bsigma^0_{\alpha}$).
\end{theorem}

Using this, one can show that any countable dense subset $A$ of a zero-dimensional perfect Polish space is $\bsigma^0_2$-complete as follows. On the one hand, since singletons are closed and $A$ is countable, $A$ is in $\bsigma^0_2$. On the other hand, it follows from the Baire category theorem that $A$ is not in $\bpi^0_2$ and hence, $A$ is $\bsigma^0_2$-complete. For example, the set $E$ of eventually constant sequences and the set $C$ of recursive sequences in the Polish space $3^{\mathbb{N}}:=\{0,1,2\}^{\mathbb{N}}$ are $\bsigma^0_2$-complete, being both countable and dense. An example of a $\bpi^0_2$-complete set that will be used in this paper is the set
\[I=\{(a_n) \in 3^{\mathbb{N}}: \forall \ell\ \exists i,j,k \geq \ell\ \ a_i=0\ \wedge\ a_j=1\ \wedge\ a_k=2 \}\]
that is, the set of sequences in $3^{\mathbb{N}}$ which contain infinitely many 0's, 1's and 2's. Observe that $I$ is in $\bpi^0_2$ and that the set of non-eventually constant sequences in the Cantor space $2^{\mathbb{N}}$, which is also $\bpi^0_2$-complete, is Wadge reduced to $I$ via the continuous map $(a_n) \mapsto (a_0,2,a_1,2,\dots)$. Thus, $I$ is $\bpi^0_2$-complete.

\subsection{The Polish space of $n$-generated marked groups} In this subsection, following the constructions in \cite{Grigorchuk84} and \cite{Champetier05}, we shall construct the Polish space of $n$-marked groups.  In order to avoid trivialities, we shall assume for the rest of the paper that $n > 1$.

Recall that an \textit{$n$-marked group} $(G,S)$ consists of a group $G$ together with a finite sequence of generators $S=(s_1,s_2,\dots,s_n)$. We allow the generating set to have repetitions and the identity element. Two $n$-marked groups $(G,(s_1,s_2,\dots,s_n))$ and $(H,(r_1,r_2,\dots,r_n))$ are said to be \textit{isomorphic} if there exists a group isomorphism $f: G \rightarrow H$ such that $f(s_i)=r_i$ for all $1 \leq i \leq n$.

Let $\mathbb{F}_n$ denote the free group with basis $\{\gamma_1,\gamma_2,\dots,\gamma_n\}$. For every $n$-marked group $(G,(s_1,s_2,\dots,s_n))$, there is a canonical epimorphism $f: \mathbb{F}_n \rightarrow G$ such that $f(\gamma_i)=s_i$ for all $1 \leq i \leq n$. Therefore, every $n$-marked group induces a normal subgroup $\mathbb{F}_n$, namely, the kernel of $f$. Conversely, every normal subgroup $N \trianglelefteq \mathbb{F}_n$ induces an $n$-marked group, namely, $(\mathbb{F}_n/N, (\gamma_1 N, \gamma_2 N, \dots, \gamma_n N))$. For the sake of readability, we shall denote the marked group $(\mathbb{F}_n/N, (\gamma_1 N, \gamma_2 N, \dots, \gamma_n N))$ by $(G_N,S_N)$.

It is readily verified that two $n$-marked groups are isomorphic if and only if they induce the same normal subgroup of $\mathbb{F}_n$. Consequently, for the purpose of coding these groups as the points of a Polish space, we may assume without loss of generality that the set of $n$-marked groups is the set
\[\spaceofg_n = \{ N \in \mathcal{P}(\mathbb{F}_n): N \trianglelefteq \mathbb{F}_n\}\]
One may endow $\mathcal{P}(\mathbb{F}_n)$ with a natural Polish topology induced by the (ultra)metric
\[ d(N,K)=\begin{cases} 0 & \text{if } N=K\\ \displaystyle 2^{-{min\{i: \mathbf{g}_i \in N \Delta K \}}} & \text{if } N \neq K\end{cases}\]
where $(\mathbf{g}_i)_{i \in \mathbb{N}}$ is a fixed enumeration of elements of $\mathbb{F}_n$. This metric is equivalent to the metric defined in \cite[2.2.a]{Champetier05}. Moreover, if one also identifies $\mathcal{P}(\mathbb{F}_n)$ with the Cantor space $\{0,1\}^{\mathbb{F}_n}$ consisting of binary sequences indexed by $\mathbb{F}_n$, this metric is compatible with the product topology on $\{0,1\}^{\mathbb{F}_n}$ where each component has the discrete topology.

It is straightforward to check that the set $\spaceofg_n$ is a closed subset of $\mathcal{P}(\mathbb{F}_n)$ and hence forms a compact zero-dimensional Polish space under the subspace topology. This topology is sometimes referred to as Grigorchuk topology or Chabauty topology. (See \cite{Cha50}.)

Next will be introduced two other metrics that induce the topology of $\spaceofg_n$. For every $N$ in $\spaceofg_n$, let $W(N,k)$ denote the set of elements of $N$ with length at most $k$. Consider the metric on $\spaceofg_n$ given by
\[ \nu(N,K)=\begin{cases} 0 & \text{if } N=K\\ \displaystyle 2^{-\max\{k: W(N,k)=W(K,k)\}} & \text{if } N \neq K \end{cases}\]
That is, the distance between the marked groups $(G_N,S_N)$ and $(G_K,S_K)$ is $2^{-k}$ if and only if $k$ is the maximum integer such that the sets of words with symbol set $\{\gamma_1^{\ \pm 1}, \dots, \gamma_n^{\ \pm 1}\}$ of length at most $k$ representing the identity in $G_N$ and $G_K$ respectively are the same.

Consider the Cayley graph $(G_N,S_N)$ as a rooted directed labeled graph where the root is the identity element and the label set is $\{\gamma_1, \gamma_2, \dots, \gamma_n\}$. Let $\mathcal{B}[N,k]$ denote the closed ball with radius $k$ centered at the root in this Cayley graph. Consider the metric on $\spaceofg_n$ given by
\[ \mu(N,K)=\begin{cases} 0 & \text{if } N=K\\ \displaystyle 2^{-\max\{k: \mathcal{B}[N,k] \cong \mathcal{B}[K,k]\}} & \text{if } N \neq K \end{cases}\]
where isomorphism is understood as the isomorphism of rooted directed labelled graphs. That is, the distance between two marked groups is $2^{-k}$ if and only if $k$ is the maximum integer such that the closed balls with radius $k$ centered at the identity element in the Cayley graphs of these marked groups are isomorphic.

A moment's thought reveals that $W(N,2k+1)=W(K,2k+1)$ if and only if $\mathcal{B}[N,k]\cong\mathcal{B}[K,k]$, that is, $\nu(N,K) \leq 2^{-(2k+1)}$ if and only if $\mu(N,K) \leq 2^{-k}$. It follows that $\mu$ and $\nu$ are equivalent metrics. Moreover, it is not difficult to check that $\mu$ and $\nu$ are compatible with the topology of $\spaceofg_n$. Since some of our arguments are easier to understand if one adopts the suitable metric, we shall be working with any of $\mu$, $\nu$ and $d$ when needed.

Let us mention some important notational remarks. For the rest of the paper, while referring to points of the space $\spaceofg_n$, we may abuse the notation and simply write an abstract marked group $(G,S)$ while we indeed refer to an element $N$ of $\spaceofg_n$ such that $(G_N,S_N)$ is isomorphic to $(G,S)$. We will also write $B_{\eta}(N,k)$ to denote the open ball centered at $N$ with radius $k$ with respect to the metric $\eta$, where $\eta$ may be any of $\nu$, $\mu$ or $d$.

Note that, the Polish space $\spaceofg_1$ is homeomorphic to $\{\frac{1}{k}: k \in \mathbb{N}^+\} \cup \{0\}$ together with the Euclidean topology, where $0$ corresponds to the infinite cyclic and $\frac{1}{k}$ correspond to the finite cyclic groups. Observe also that the Polish space $\spaceofg_n$ embeds into $\spaceofg_{n+1}$ via the map $(G,(s_1,\dots,s_n)) \mapsto (G,(s_1,\dots,s_n,1_G))$.

Having introduced three compatible metrics with the topology of $\spaceofg_n$, we shall next recall some simple facts regarding limits of various sequences of marked groups in $\spaceofg_n$. Most of these appear either explicitly or implicitly in \cite{Grigorchuk84}. For the self-containment of this paper as a survey, we shall provide the proofs of these facts, some of which appear to not have been explicitly written.

Let $P$ be a property of groups. Recall that a group $G$ is said to be \textit{fully residually $P$} if for every positive integer $m$ and any distinct elements $g_1,\ldots,g_m \in G$, there exist a group $H$ with the property $P$ and a homomorphism $\varphi :G \to H$ such that $\varphi(g_1),\ldots,\varphi(g_m)\in H$ are distinct. It is readily verified that if the property $P$ is preserved under (finite) direct products, then the class of fully residually $P$ groups coincides with the class of \textit{residually $P$ groups}, that is, those groups that satisfy the condition stated above for $m=2$. It turns out that fully residually $P$ groups are limits of groups with the property $P$ in $\spaceofg_n$.

\begin{proposition}\cite{Champetier05}\label{residuallypconverge}
Suppose that $P$ is a property of groups which is inherited by subgroups. Let $N \in \spaceofg_n$ be such that $G_N$ is fully residually $P$. Then $(G_N,S_N)$ is a limit of marked groups with the property $P$.
\end{proposition}
\begin{proof}
Let $m$ be a positive integer. Since $G_N$ is fully residually $P$, there exist a group $H_m$ with the property $P$ and a homomorphism $\varphi_m: G_N \rightarrow H_m$ such that the images of the elements in $\mathcal{B}[N,m]$ under $\varphi_m$ are all distinct. Now set $K_m$ to be the subgroup $\langle \varphi_m[S_N] \rangle \leqslant H_m$. Then $\mathcal{B}[N,m]$ embeds into the Cayley graph of $(K_m,\varphi_m[S_N])$. Moreover, by the inheritance of $P$ by subgroups, $K_m$ has the property $P$. It can be checked that $$\mu((G_N,S_N),(K_m,\varphi_m[S_N])) \leq 2^{-m}$$ This shows that $G_N$ is a limit of groups with the property $P$.\end{proof}

We will now prove that any group in $\spaceofg_n$ is a limit of finitely presented groups and that finitely presented groups have neighborhoods which entirely consists of their quotients.

\begin{proposition} \cite{Grigorchuk84} \label{finitelypresentedconverge} For every $N \in\spaceofg_n$, there exists a sequence of finitely presented groups with limit $(G_N,S_N)$.\end{proposition}
\begin{proof} Let $N$ be $\spaceofg_n$ and $(w_i)_{i \in  \mathbb{N}}$ be an enumeration of the words in $N$. For every $m \in \mathbb{N}$, set $N_m$ to be the normal closure of $\{w_i: 0 \leq i \leq m\}$ in $\mathbb{F}_n$, that is, $N_m \in \spaceofg_n$ is such that
\[ G_{N_m}=\langle \gamma_1,\dots,\gamma_n |\ w_i,\ 0 \leq i \leq m \rangle\]
It is straightforward to check that, with respect to $\nu$, the sequence $(N_m)_{m \in \mathbb{N}}$ converges to $N$.
\end{proof}

\begin{proposition}\cite[Lemma 2.3]{Champetier05}\label{finitelypresentedquotient} Let $N \in \spaceofg_n$ be such that $G_N$ is finitely presented. Then there exists $\epsilon>0$ such that every element of $B_{\mu}(N,\epsilon)$ is a quotient of $G_N$.\end{proposition}

\begin{proof} Let $\langle\gamma_1\ldots,\gamma_n \mid w_i,\; 0 \leq i \leq m  \rangle $ be a finite presentation of $G_N$. Set $k$ to be $\max\{ |w_i| , 0 \leq i \leq m\} $ and choose $0 < \epsilon < 2^{-k}$. Then, by definition, for every $M \in B_{\mu}(N,\epsilon)$, the words $w_i$ where $1 \leq i \leq m$ are relations of $G_M$ and hence, $G_M$ is a quotient of $G_N$.\end{proof}

Before concluding this section, we shall introduce a definition that is used to characterize groups which are limits of groups with a specific property. Let $P$ be a property of groups. A group $G$ is called \textit{locally embeddable into groups with the property $P$} (in short, LEP) if for every finite subset $E\subseteq G$ there exist a group $H$ with the property $P$ and a \textit{local embedding on $E$ to $H$}, that is, a map $\phi: E \cup (E \cdot E) \to H$ such that $\phi$ is injective on $E$ and $\phi(gh)=\phi(g)\phi(h)$ for every $g,h\in E$.

%Since for any $n,m\ge 1$, the spaces $\mathcal{G}_n$ and $\mathcal{G}_m$ are locally homeomorphic by
%a homeomorphism that preserves isomorphism (see \cite[Lemma 1]{CGP07}), being $LEP$ is a well defined %property of a group $G$.

\begin{proposition}\cite{Ver97}\label{lep} Let $P$ be a property of groups which is inherited by subgroups. Then, $G_N \in \mathcal{G}_n$ is a limit of marked groups in $\mathcal{G}_n$ with the property $P$ if and only if
$G_N$ is $LEP$.\end{proposition}

\begin{proof} 
Suppose $G_N$ is $LEP$. As $G_N$ is countable, there exist finite subsets $E_1 \subseteq E_2 \subseteq \dots$ such that $G_N=\bigcup_{k \in \mathbb{N}^+} E_k$. Since $G$ is $LEP$, for every $k \in \mathbb{N}^+$, we have a local embedding $\phi_k : E_k \to F_k$ to a group $F_k$ with the property $P$. Let $k$ be large enough so that $S_N \subseteq E_k$. Set  $H_k=\langle\phi_k[S] \rangle\le F_k$. Then the marked group
$(H_k,\phi[S_k])$ belongs to $\mathcal{G}_n$ and, by the inheritance of $P$ by subgroups, $H_k$ has the property $P$.

We claim that the sequence $(H_k,\phi_k[S])$ converges to $(G_N,S_N)$. Given $m>0$, let $k$ be large enough so that $E_k $ contains $\mathcal{B}[N,m]$. Thus, $\mathcal{B}[N,m]$ embeds into the Cayley graph of $(H_k,\phi_k[S])$. Thus we have 
$$\mu((G_N,S_N),(H_k,\phi_k[S_N])) \leq 2^{-m}$$ 
Conversely, suppose that there is a sequence $(H_k,S_k)=(G_{N_k},S_{N_k})$ of groups with the property $P$ in $\mathcal{G}_n$ converging to $(G_N,S_N)$. Given any finite subset $E$ of $G$, let $m$ be large enough so that $E\cdot E\subseteq \mathcal{B}[N,m]$ and let $k$ be large 
enough so that $\mathcal{B}[N,m]\cong \mathcal{B}[N_k,m]$. Using this isomorphism, one can easily define a local embedding $\phi:E\cup (E \cdot E) \to H_k$.\end{proof}

\section{Group theoretic properties as topological properties}

In this section, we shall analyze the descriptive complexity of subsets of $\spaceofg_n$ defined by group theoretic properties. In other words, we will determine at which level of the Borel hierarchy the set $$\{N \in \spaceofg_n: G_N \text{ is } P\}$$ resides for various properties $P$.

\subsection{Definability of properties} In order to formalize the notion of ``a group being $P$" for a property $P$, we are going to use the first-order infinitary logic $\mathcal{L}_{\omega_1 \omega}$ in the language of group theory together with constant symbols for every element of $\mathbb{F}_n$. Recall that, one is allowed to have finite quantifications but countable conjunctions and disjunctions in $\mathcal{L}_{\omega_1 \omega}$. Let $\varphi$ be an $\mathcal{L}_{\omega_1 \omega}$-sentence. We shall say that a property $P$ is defined by $\varphi$ in the space $\spaceofg_n$ if $$\{N \in \spaceofg_n: G_N \models \varphi\}=\{N \in \spaceofg_n: G_N \text{ is } P\}$$
where the constant symbols for elements of $\mathbb{F}_n$ are interpreted in the obvious way, i.e. $g$ is interpreted as $gN$ in the quotient group $G_N=\mathbb{F}_n/N$. A simple but important observation is that, since the underlying structures are countable and each element is named by a constant, any sentence is equivalent to a quantifier-free sentence in each $G_N$, as one can replace universal quantifiers by appropriate conjunctions and existential quantifiers by appropriate disjunctions, using the constant symbols.

Before we proceed, let us mention why restricting our attention to properties that are defined by some $\mathcal{L}_{\omega_1 \omega}$-sentence is sufficient. We are mostly interested in algebraic properties that define Borel sets in $\spaceofg_n$. We shall see later that each Borel set in $\spaceofg_n$ can be defined by some quantifier-free $\mathcal{L}_{\omega_1 \omega}$-sentence in this augmented language.

\subsection{Some group theoretic properties} In this subsection, we shall determine the exact descriptive complexity of being finite, finitely presented, abelian, nilpotent and torsion-free. The results in this section are all folklore.

\subsubsection{Being finite.} Let $N \in \spaceofg_n$ be such that $G_N$ is finite. Set $k$ to be the radius of the Cayley graph of $(G_N,S_N)$. Since the closed ball $\mathcal{B}[N,k]$ completely determines the group operation, we have that $B_{\mu}(N,2^{-(k+1)})=\{N\}$. It follows that being finite is a $\bsigma^0_1$-property (and that finite groups are isolated.)

Moreover, it is easily seen that the limit of the sequence of the finite cyclic groups $(\mathbb{Z}/k\mathbb{Z},(1))$ is $(\mathbb{Z},(1))$ and hence, being finite is not a $\bpi^0_1$-property. Therefore, the set of finite groups in $\spaceofg_n$ is $\bsigma^0_1$-complete.

\subsubsection{Being finitely presented.} Since there are countably many finitely presented marked groups, the set of finitely presented groups in $\spaceofg_n$ is clearly in $\bsigma^0_2$. On the other hand, it follows from Proposition \ref{finitelypresentedconverge} that the isolated points of $\spaceofg_n$ are finitely presented and that the set of finitely presented groups is dense in the perfect part of $\spaceofg_n$. It then follows from a Baire category argument that this set cannot be in $\bpi^0_2$. Thus the set of finitely presented groups in $\spaceofg_n$ is in $\bsigma^0_2$-complete.

\subsubsection{Being abelian.} It is clear that the property of being abelian is defined by the sentence
\[ \bigwedge_{1 \leq i,j \leq n} [\gamma_i,\gamma_j]=e\]
Since each sentence of the form $[\gamma_i,\gamma_j] = e$ defines a $\bdelta^0_1$-subset of $\spaceofg_n$, we have that the set of abelian groups in $\spaceofg_n$ is $\bdelta^0_1$.

This fact can also be observed as follows. For every $K,N \in \spaceofg_n$, if $\nu(K,N) < 2^{-3}$, then $W(N,4)=W(K,4)$, in which case the sentences of the form $[\gamma_i,\gamma_j]=e$ that are satisfied in $G_N$ and $G_K$ are the same. Thus, any $\nu$-open ball with radius $2^{-3}$ of an abelian (respectively, non-abelian) group consists of abelian (respectively, non-abelian) groups.

\subsubsection{Being nilpotent.} Note that the property of being nilpotent is defined by the sentence
\[ \bigvee_{k \in \mathbb{N}^+} \left( \bigwedge_{1 \leq i_1,\dots,i_{k+1} \leq n} [\gamma_{i_1},\gamma_{i_2},\dots,\gamma_{i_{k+1}}]=e \right) \]

Each sentence of the form $[\gamma_{i_1},\gamma_{i_2},\dots,\gamma_{i_{k+1}}]=e$ defines a $\bdelta^0_1$-subset of $\spaceofg_n$ and hence, being nilpotent of class at most $k$ is a $\bdelta^0_1$-property and being nilpotent is a $\bsigma^0_1$-property. On the other hand, since non-abelian free groups are fully residually finite-$p$ (for any prime $p$) \cite{MHall50}, by Proposition \ref{residuallypconverge}, there exists a sequence of  finite nilpotent  groups converging to a non-abelian free group. This shows that being nilpotent is not a $\bpi^0_1$-property and hence, the set of nilpotent groups in $\spaceofg_n$ is $\bsigma^0_1$-complete.

\subsubsection{Being torsion-free.} By definition, the set of torsion-free groups in $\spaceofg_n$ is defined by the sentence
\[ \bigwedge_{\delta \in \mathbb{F}_n}\left(e=\delta\ \vee \left(\bigwedge_{k \in \mathbb{N}^+} \delta^k \neq e\right)\right)\]
Since each sentence of the form $\delta^k \neq e$ defines a $\bdelta^0_1$-subset of $\spaceofg_n$, the set of torsion-free groups in $\spaceofg_n$ is in $\bpi^0_1$. However, as before, since a sequence of non torsion-free (in particular, finite) groups can converge to a non-abelian free group in $\spaceofg_n$, this set is not in $\bsigma^0_1$. Thus, the set of torsion-free groups in $\spaceofg_n$ is $\bpi^0_1$-complete.

\subsection{Being P v. being virtually P}

Having set up a formal framework for definability of properties using an infinitary logic, as an illustration of how this can be used to prove general facts regarding descriptive complexity of properties, we will now analyze the relationship between being $P$ and being virtually $P$ in this subsection. Recall that given a property $P$ of groups, a group $G$ is said to be \textit{virtually $P$} if $G$ has a finite index subgroup with the property $P$.

Before we prove the main result of this subsection, let us give a stratification of a special subclass of quantifier-free formulas of $\mathcal{L}_{\omega_1 \omega}$. This stratification will allow us to prove that any Borel set in $\spaceofg_n$ is defined by some $\mathcal{L}_{\omega_1 \omega}$ sentence, a fact promised earlier. We would like to remind the reader that we have a constant symbol for each element of $\mathbb{F}_n$, which is necessary for this fact to hold.

Let $\Sigma^{\text{qf}}_0=\Pi^{\text{qf}}_0$ be the set of quantifier-free $\mathcal{L}_{\omega_1 \omega}$-sentences that are finite Boolean combinations of atomic sentences. For every $1 \leq \alpha < \omega_1$, we inductively define $\Sigma^{\text{qf}}_{\alpha}$ (respectively, $\Pi^{\text{qf}}_{\alpha}$) to be the set of countable disjunctions (respectively, conjunctions) of sentences in $\bigcup_{0 \leq \beta < \alpha} \Pi^{\text{qf}}_{\beta}$ (respectively, $\bigcup_{0 \leq \beta < \alpha} \Sigma^{\text{qf}}_{\beta}$.) Observe that the sets of the form
\begin{align*}\{N \in \spaceofg_n: g \in N\}&=\{N \in \spaceofg_n: G_N \models\ g=e\}\\ \{N \in \spaceofg_n: g \notin N\}&=\{N \in \spaceofg_n: G_N \models\ g \neq e\}
\end{align*}
where $g$ ranges over $\mathbb{F}_n$, form a clopen subbase for the topology of $\spaceofg_n$. Thus the sets defined by finite Boolean combinations of atomic sentences form a clopen basis for the topology of $\spaceofg_n$. It follows that any set in $\bsigma^0_1$ is defined by some sentence in $\Sigma^{\text{qf}}_1$ and any set in $\bpi^0_1$ is defined by some sentence in $\Pi^{\text{qf}}_1$. It can now be proven by induction that every set in $\bsigma^0_{\alpha}$ (respectively, $\bpi^0_{\alpha}$) is defined by some quantifier-free sentence in $\Sigma^{\text{qf}}_\alpha$ (respectively, $\Pi^{\text{qf}}_\alpha$.) Conversely, another inductive argument implies that any sentence in $\Sigma^{\text{qf}}_\alpha$ (respectively, $\Pi^{\text{qf}}_\alpha$) defines a set in $\bsigma^0_{\alpha}$ (respectively, $\bpi^0_{\alpha}$.)

We would like to note that a better stratification of a broader class formulas of $\mathcal{L}_{\omega_1 \omega}$ can be given. For example, see \cite[Section 3.2]{Osin20} for such a stratification and \cite[Lemma 5.3]{Osin20} for the counterpart of the fact we proved above. That said, our somewhat primitive stratification is going to suffice for our purposes.

We have the following proposition, whose proof will be partly used later to analyze the descriptive complexity of residual finiteness.

\begin{lemma}\label{virtualplemma} Suppose that ``being $P$" is in $\bsigma^0_{\alpha}$ (respectively, $\bpi^0_{\alpha}$) for all $n \in \mathbb{N}^+$. Then, for each integer $n >1$, ``being virtually $P$" is in
\begin{itemize}
\item $\bsigma^0_3$ whenever $\alpha < 3$.
\item $\bsigma^0_{\alpha}$ (respectively, $\bsigma^0_{\alpha+1}$) whenever $\alpha \geq 3$.
\end{itemize}
\end{lemma}
\begin{proof} As mentioned before, for each $n \in \mathbb{N}^+$, we can choose some quantifier-free sentence $\varphi_n$ in $\Sigma^{\text{qf}}_{\alpha}$ (respectively, $\Pi^{\text{qf}}_{\alpha}$) which defines ``being $P$" in $\spaceofg_n$. Let $n >1$ and fix an enumeration $(A_i)_{i \in \mathbb{N}}$ of non-empty finite subsets of $\mathbb{F}_n$. Fix $i \in \mathbb{N}$, set $p_i=|A_i|$ and let $A_i=\{a_1,\dots,a_{p_i}\}$.

Consider the sentence $\phi_{A_i}$ obtained from $\varphi_{p_i}$ after
\begin{itemize}
\item replacing each $\gamma_j$ by the word $a_j$ for every $1 \leq j \leq p_i$, and
\item replacing the other elements of $\mathbb{F}_{p_i}$ by words in $\langle A_i \rangle$ accordingly.
\end{itemize}
Observe that an atomic sentence is turned into another atomic sentence $\mathcal{L}_{\omega_1 \omega}$. Consequently, one can show by induction on the complexity of sentences that $\phi_{A_i}$ is in $\Sigma^{\text{qf}}_{\alpha}$ (respectively, $\Pi^{\text{qf}}_{\alpha}$) and hence the set $\{N \in \spaceofg_n: G_N \models \phi_{A_i}\}$ is also in $\bsigma^0_{\alpha}$ (respectively, $\bpi^0_{\alpha}$.)

For each $N \in \spaceofg_n$, let $H_{N,A_i}$ denote the finitely generated subgroup 
\[\langle A_i \rangle N/N=\langle a_j N: 1 \leq j \leq p_i\rangle \leqslant G_N\]
Let $N \in \spaceofg_n$. Then the marked group $(H_{N,A_i},(a_1 N,\dots,a_{p_i}N))$ is isomorphic to the marked group
$$ (G_{K_{N,A_i}},S_{K_{N,A_i}})=(\mathbb{F}_{p_i}/K_{N,A_i}, (\gamma_1 K_{N,A_i}, \gamma_2 K_{N,A_i}, \dots, \gamma_{p_i} K_{N,A_i}))$$
where $K_{N,A_i}$ is the kernel of the canonical surjection from $\mathbb{F}_{p_i}$ to $H_{N,A_i}$ extending the map $\gamma_j \mapsto a_j N$. Moreover, for every word $w$ with symbols from $\{\gamma_1^{\ \pm 1}, \dots, \gamma_{p_i}^{\ \pm 1}\}$, the word $w$ is in $\widehat{N}$ if and only if the word obtained from $w$ after replacing each $\gamma_j$ by $a_j$ is in $N$. It follows that the set
\begin{align*}
\{N \in \spaceofg_n: H_{N,A_i} \text{ is } P\}&=\{N \in \spaceofg_n: G_{K_{N,A_i}} \text{ is } P\}\\
&=\{N \in \spaceofg_n: G_{K_{N,A_i}} \models \varphi_{p_i}\}\\
&=\{N \in \spaceofg_n: G_N \models \phi_{A_i}\}
\end{align*}
is in $\bsigma^0_{\alpha}$ (respectively, $\bpi^0_{\alpha}$.) Now, we want to understand the complexity of the set of $N$'s for which $H_{N,A_i}$ is of finite index in $G_N$. Note that, for every integer $j \geq 2$, we have $[G_N:H_{N,A_i}] < j$ if and only if $[\mathbb{F}_{n}: \langle A_i \rangle N] < j$. Thus

\begin{align*}
&\ \ \ \ \{N \in \spaceofg_n: [G_N:H_{N,A_i}] < j\}\\&=\left\{N \in \spaceofg_n: [\mathbb{F}_{n}: \langle A_i \rangle N] < j\right\}\\
&=\left\{N \in \spaceofg_n: \forall g_1,\dots,g_j \in \mathbb{F}_n\ \ \exists k \neq \ell \in \{1,\dots,j\}\ \ \exists \delta_{k\ell} \in \langle A_i \rangle\ \delta^{-1}_{k\ell} g^{-1}_k g_{\ell} \in N\right\}\\
&=\bigcap_{g_1,\dots,g_j \in \mathbb{F}_n} \bigcup_{1 \leq k \neq \ell \leq j} \bigcup_{\delta_{k\ell} \in \langle A_i \rangle} \left\{N \in \spaceofg_n: \delta^{-1}_{k\ell} g^{-1}_k g_{\ell} \in N\right\}
\end{align*}
Since the sets $\left\{N \in \spaceofg_n: \delta^{-1}_{k\ell} g^{-1}_k g_{\ell} \in N\right\}$ are clopen for every $k$ and $\ell$, the set defined above is in $\bpi^0_2$.

We claim that, for every $N \in \spaceofg_n$, the group $G_N$ has a subgroup of finite index if and only if there exists $i \in \mathbb{N}$ such that $[G_N: H_{N,A_i}] \leq |A_i|$. Suppose that $L \leqslant G_N$ is of finite index. So $L = K / N$ for some $K \geqslant N$ with $[\mathbb{F}_n:K]=[G_N:L]$. Therefore $K=\langle A_i \rangle = \langle A_i \rangle N$, that is, $L=H_{N,A_i}$ for some $i \in \mathbb{N}$. By Schreier index formula
\[ |A_i| \geq \text{rank}(K)=[\mathbb{F}_n:K] (n-1)+1 \geq [\mathbb{F}_n:K]=[G_N:L]\]
It follows that
\begin{align*}
&\{N \in \spaceofg_n: G_N \text{ is virtually } P\}=\\
&\bigcup_{i=0}^{\infty}\ \{N \in \spaceofg_n: H_{N,A_i} \text{ is } P\ \text{ and }\ [G_N: H_{N,A_i}] \leq |A_i|\}=\\
&\bigcup_{i=0}^{\infty}\ \{N \in \spaceofg_n: H_{N,A_i} \text{ is } P\} \cap \{N \in \spaceofg_n: [G_N: H_{N,A_i}] \leq |A_i|\}
\end{align*}
which can easily be checked to be in $\bsigma^0_3$ whenever $\alpha < 3$, and in $\bsigma^0_{\alpha}$ (respectively, $\bsigma^0_{\alpha+1}$) whenever $\alpha \geq 3$.\end{proof}
 
\subsection{Solvability, periodicity and some growth properties}

In this subsection, we shall determine the exact descriptive complexity of solvability, periodicity and having various growth properties. For a general background regarding the notion of the growth of a finitely generated group, we refer the reader to \cite{Mann12}.

For later use, we shall first observe that the set of groups with polynomial growth in $\spaceofg_n$ is in $\bsigma^0_1$ (and indeed is $\bsigma^0_1$-complete.) By Gromov's theorem \cite{Gromov81}, the class of groups with polynomial growth coincides with the class of finitely generated virtually nilpotent groups. On the other hand, finitely generated virtually nilpotent groups are finitely presented and hence, by Proposition \ref{finitelypresentedquotient}, every virtually nilpotent group in $\spaceofg_n$ has a neighborhood which entirely consists of its quotients, which are also virtually nilpotent. Thus, the set of virtually nilpotent (and hence, polynomial growth) groups in $\spaceofg_n$ is in $\bsigma^0_1$. On the other hand, this set is not in $\bpi^0_1$ since, as before, there exists a sequence of polynomial growth (in fact, finite) groups whose limit is an exponential growth (in fact, non-abelian free) group. Thus, the set of groups with polynomial growth in $\spaceofg_n$ is $\bsigma^0_1$-complete.

We are now ready to state our main theorem.

\begin{theorem} Let $n > 1$ be an integer. Then
\begin{itemize}
\item[a.] the set of solvable groups in $\spaceofg_n$ is $\bsigma^0_2$-complete.
\item[b.] the set of periodic groups in $\spaceofg_n$ is $\bpi^0_2$-complete. 
\item[c.] the set of groups with exponential growth in $\spaceofg_n$ is $\bsigma^0_2$-complete. 
\item[d.] the set of groups with intermediate growth in $\spaceofg_n$ is $\bpi^0_2$-complete. 
\item[e.] the set of groups with decidable word problem in $\spaceofg_n$ is $\bsigma^0_2$-complete.
\end{itemize} 
\end{theorem}
\begin{proof} Observe that the property of being solvable is defined by the sentence
\[ \bigvee_{k \in \mathbb{N}^+} \left(\bigwedge_{\delta \in \mathbb{F}_n^{(k)}} \delta = e\right) \] 
Each sentence $\delta = e$ defines a $\bdelta^0_1$-subset of $\spaceofg_n$ and hence, being solvable of degree at most $k$ is a $\bpi^0_1$-property and being solvable is a $\bsigma^0_2$-property.

By definition, the property of being periodic is defined by the sentence
\[ \bigwedge_{\delta \in \mathbb{F}_n} \left(\bigvee_{k \in \mathbb{N}^+} \delta^k=e\right)\]
Since each sentence $\delta^k=e$ defines a $\bdelta^0_1$-subset of $\spaceofg_n$, being periodic is a $\bpi^0_2$-property.

A marked group $(G_N,S_N)$ has exponential growth if there exist a real (equivalently, rational) number $a>1$ with
\[  \Gamma_N(x) \geq a^x \]
for every $x \in \mathbb{N}^+$, where the growth function $\Gamma_N(x)$ is the number of vertices in $\mathcal{B}[N,x]$. In other words, the set of groups in $\spaceofg_n$ with exponential growth is
\[ \bigcup_{a \in \mathbb{Q}_{>1}}\ \bigcap_{x \in \mathbb{N}^+} \left\{N \in \spaceofg_n: \Gamma_N(x) \geq a^x \right\}\]
Observe that, for every fixed rational $a>1$ and fixed $x \in \mathbb{N}^+$, the inner-most set defines a $\bdelta^0_1$-set because $\Gamma_N(x)=\Gamma_K(x)$ whenever $\mu(N,K) \leq 2^{-x}$. Hence, having exponential growth is a $\bsigma^0_2$-property. (We would like to note that, the argument above indeed shows that, given two functions $f,g: \mathbb{N}^+ \rightarrow \mathbb{R}$, the set of groups in $\spaceofg_n$ with growth function $\Gamma_N(x)$ satisfying $f(x) \leq \Gamma_N(x) \leq g(x)$ for all $x \in \mathbb{N}^+$ is in $\bsigma^0_2$.)

By definition, the set of groups with intermediate growth in $\spaceofg_n$ is the intersection of the complements of the sets of exponential growth groups and polynomial growth groups in $\spaceofg_n$, which are in $\bpi^0_2$ and $\bpi^0_1$ respectively. It follows that this set is in $\bpi^0_2$. The set of groups with decidable word problem is countable and hence is in $\bsigma^0_2$.

We shall next show that the sets of solvable groups, periodic groups, groups with exponential growth, groups with intermediate growth and groups with decidable word problem are complete in their respective point classes.

Let $E$ and $C$ denote the set of eventually constant sequences and the set of recursive sequences in $3^{\mathbb{N}}$ respectively. Let $I$ denote the set of sequences in $3^{\mathbb{N}}$ which contain infinitely many 0's, 1's and 2's. We have observed in Section 2.1 that $E$ and $C$ are $\bsigma^0_2$-complete and $I$ is $\bpi^0_2$-complete. It follows from Grigorchuk's construction \cite{Grigorchuk84} and its slight modification in \cite{BenliGrigorchuk14} that, for each $n\ge 2$, there is a continuous function $f:3^{\mathbb{N}} \to \m{G}_n$ such that
\begin{itemize}
\item $f[E]$ consists of solvable groups of exponential growth,
\item $f[3^{\mathbb{N}}-E]$ consists of intermediate growth (and hence, non-solvable) groups,
\item $f(\om)$ is periodic if and only if $\om \in I$,
\item $f(\alpha)$ has decidable word problem if and only if $\alpha \in C$.
\end{itemize}
See Appendix A for the details of this construction. In other words, $E$ is Wadge reducible to both the set of solvable groups and the set of groups with exponential growth, $C$ is Wadge reducible to the set of groups with decidable word problem, $3^{\mathbb{N}}-E$ is Wadge reducible to the set of groups with intermediate growth and $I$ is Wadge reducible to the set of periodic groups in $\spaceofg_n$. This completes the proof.\end{proof}

\subsection{More group theoretic properties} In this subsection, we shall provide some upper and lower bounds for the descriptive complexities of the properties of being simple, amenable, residually finite, Hopfian and co-Hopfian.

\subsubsection{Being simple} Recall that a group is simple if and only if the normal closure of any non-identity element is the whole group. In other words, the property of being simple is defined by the sentence
\[ \bigwedge_{\gamma \in \mathbb{F}_n} \left(e=\gamma\ \vee  \bigwedge_{\delta \in \mathbb{F}_n} \bigvee_{k \in \mathbb{N}^+} \left(\bigvee_{\delta_1,\dots,\delta_k \in \mathbb{F}_n} \bigvee_{\epsilon_1,\dots,\epsilon_k \in \{-1,1\}}  \delta = \left(\gamma^{\delta_1}\right)^{\epsilon_1} \dots \left(\gamma^{\delta_k}\right)^{\epsilon_k} \right)\right) \]
Observe that each sentence $\delta = \left(\gamma^{\delta_1}\right)^{\epsilon_1} \dots \left(\gamma^{\delta_k}\right)^{\epsilon_k}$ defines a $\bdelta^0_1$-subset of $\spaceofg_n$ and hence, the set of simple groups in $\spaceofg_n$ is in $\bpi^0_2$.  We conjecture that this set is indeed $\bpi^0_2$-complete. While we were not able to prove this, one can easily show that this set is not in $\bsigma^0_1$ or $\bpi^0_1$ as follows.

Since the sequence of finite simple groups $(\mathbb{Z}_{p_k},(1))$, where $p_k$ is the $k$-th prime, converges to the group $(\mathbb{Z},(1))$, the set of simple groups in $\spaceofg_n$ is not in $\bpi^0_1$. To show that this set is not in $\bsigma^0_1$, let $G=\langle \gamma_1,\dots,\gamma_n |\ w_i,\ i \in \mathbb{N} \rangle$ be a simple group which is not finitely presented. (For example, in \cite{Cam53}, the author constructs continuum many $2$-generated infinite simple groups, which implies via a counting argument that there exist $2$-generated infinite simple groups that are not finitely presented.) For each $k \in \mathbb{N}$, set $$G_k = \langle \gamma_1,\dots,\gamma_n |\ w_i,\ 0 \leq i \leq k \rangle$$ The relations of $G_k$ are also relations of $G$ and consequently, there exist surjective homomorphisms $\varphi_k: G_k \rightarrow G$. Since $G$ is not finitely presented, the maps $\varphi_k$ cannot be isomorphisms and hence, are not injections. It follows that $G_k$ is not simple for any $k \in \mathbb{N}$. On the other hand, as in the proof of Proposition \ref{finitelypresentedconverge}, the groups $G_k$ converge to $G$ in $\spaceofg_n$ which implies that the set of simple groups in $\spaceofg_n$ is not in $\bsigma^0_1$.

\subsubsection{Being amenable.} Recall that, by F\o lner's condition, a countable group $G$ is amenable if and only if for every finite subset $\mathbf{K} \subseteq G$ and every positive integer $m$, there exists a non-empty finite subset $\mathbf{F} \subseteq G$ such that 
\[ \displaystyle \frac{| g\mathbf{F} \Delta \mathbf{F}| }{|\mathbf{F}|} \leq \frac{1}{m}\]
for every $g \in \mathbf{K}$. Translating this condition to our setting, we have that the set of amenable groups in $\spaceofg_n$ is
\[ \bigcap_{K \in \mathcal{F}} \bigcap_{m \in \mathbb{N}^+} \bigcup_{F \in \mathcal{F}} \bigcap_{\gamma \in K} \left\{N \in \spaceofg_n: \frac{|(\gamma N \cdot FN) \Delta FN|}{|FN|} \leq \frac{1}{m} \right\}\]
where $FN=\{\alpha N : \alpha \in F\}$ and $\mathcal{F}$ denotes the set of finite subsets of $\mathbb{F}_n$. Observe that, for every fixed values of the parameters $K,m,F$ and $\gamma$, the inner-most set is open. To see this, fix these parameters and choose $k \geq 1$ big enough so that $$F^{-1}F,\ F^{-1} \gamma F \subseteq \{\mathbf{g}_i: 1 \leq i \leq k\}$$ If $N_1, N_2 \in \spaceofg_n$ and $d(N_1,N_2)<2^{-k}$, then we have that $(F^{-1} F) \cap N_1 = (F^{-1} F) \cap N_2$ and $(F^{-1} \gamma F) \cap N_1 = (F^{-1} \gamma F) \cap N_2$.

The first equality implies that the map $\alpha N_1 \mapsto \alpha N_2$ from $F N_1$ to $F N_2$ and the map $\gamma \alpha N_1 \mapsto \gamma \alpha N_2$ from $\gamma F N_1$ to $\gamma F N_2$ are well-defined and bijective. The second equality implies that these maps agree on $F N_1 \cap \gamma F N_1$. It follows that $|F N_1|= |F N_2|$ and $|\gamma F N_1 \Delta F N_1|=|\gamma F N_2 \Delta F N_2|$ whenever $d(N_1,N_2)<2^{-k}$. So, for any sufficiently close $N_1, N_2 \in \spaceofg_n$, the corresponding inequalities for $N_1$ and $N_2$ will be both true or both false, which implies that the inner-most set is open.

It then follows that the set of amenable groups in $\spaceofg_n$ is in $\bpi^0_2$. We currently do not know whether or not this set is $\bpi^0_2$-complete. We strongly suspect that this is the case. However, it is clear that this set is not in $\bsigma^0_1$  since finite groups can converge to a non-abelian free group. Also, it is not  in $\bpi^0_1$, since non-amenable groups can converge to a solvable group such as $\mathbb{Z}_2 \wr \mathbb{Z}$. More generally, it was proven in \cite{Bs1980} that every infinitely presented metabelian group enjoys this property.
 (Also, see \cite{Bgh13} for more  examples of amenable groups as limits of non-amenable groups.)

\subsubsection{Being residually finite} Recall that a finitely generated group $G$ is residually finite if and only if for every non-identity $g \in G$ there exists a finite index (and hence, finitely generated) subgroup $H$ of $G$ such that $g \notin H$. Thus, retaining the notation in the proof of Lemma \ref{virtualplemma}, one can write the set of residually finite groups in $\spaceofg_n$ as
\[\bigcap_{\gamma \in \mathbb{F}_n} \left(\{N \in \spaceofg_n:\ \gamma \in N\} \cup \bigcup_{i=0}^{\infty}\ \{N \in \spaceofg_n:\ \gamma N \notin H_{N,A_i}\ \wedge\ [G_N:H_{N,A_i}]<\infty\}\right)\]
The condition that $\gamma N \notin H_{N,A_i}$ can be expressed as
\[ \forall \delta \in \langle A_i \rangle\ \ \gamma^{-1}\delta \notin N\]
Therefore, the set of residually finite groups in $\spaceofg_n$ is
\[\bigcap_{\gamma \in \mathbb{F}_n}\ \left(\{N \in \spaceofg_n:\ \gamma \in N\} \cup \bigcup_{i=0}^{\infty}\ \bigcup_{j=2}^\infty\ \bigcap_{\delta \in \langle A_i\rangle}\ \{N \in \spaceofg_n:\ \gamma^{-1}\delta \notin N\ \wedge\ [G_N:H_{N,A_i}]< j\}\right)\]
It can be seen from the proof of Lemma \ref{virtualplemma} that, for every fixed value of the parameters, the innermost set on the right is in $\bpi^0_2$ and hence the set of residually finite groups in $\spaceofg_n$ is in $\bpi^0_4$. Given the inelegancy of this result, the authors hope that this bound can be improved.

Let $\operatorname{Sym}_f(\mathbb{Z})$  be the group of finitely supported permutations of $\mathbb{Z}$. Consider the group $G=\operatorname{Sym}_f(\mathbb{Z}) \rtimes \mathbb{Z}$ where $\mathbb{Z}$ acts on
 $\operatorname{Sym}_f(\mathbb{Z})$ by the left shift. The group $G$ is $2$-generated and is not residually finite since it contains an infinite simple subgroup. Given any finite subset $E\subseteq G$, one can choose a large enough $k\ge 1$ so that there is a local embedding on $G\to S_k \rtimes \mathbb{Z}_k$ on $E$. It follows by Proposition \ref{lep} that $G$ is a limit of finite groups. Thus, the set of residually finite groups in $\spaceofg_n$ is not in $\bpi^0_1$.
 
 On the other hand, it is proven in \cite{Stalder06} that the Baumslag-Solitar groups $B(m,n)$ converges to $\mathbb{F}_2$ in $\mathcal{G}_2$ as $|m|,|n| \rightarrow \infty$. Since the groups $B(m,n)$ are not residually finite for distinct $m,n > 1$, the set of residually finite groups in $\spaceofg_n$ is not in $\bsigma^0_1$. We do not know whether or not this set is in $\bsigma^0_2$ or $\bpi^0_2$.
 
\subsubsection{Being Hopfian and co-Hopfian} Recall that a group is Hopfian (respectively, co-Hopfian) if every epimorphism (respectively, monomorphism) from this group to itself is an isomorphism.

Let $F \subseteq \mathbb{F}_n$ be a subset of size $n$, say, $F=\{w_1,\dots,w_n\}$. For each $g \in \mathbb{F}_n$, let $g_* \in \mathbb{F}_n$ denote the image of the canonical homomorphism from $\mathbb{F}_n$ to $\mathbb{F}_n$ sending each $\gamma_i$ to $w_i$. Then the map from $G_N$ to $G_N$ given by $gN \mapsto g_* N$ is well-defined whenever $h^{-1} g \in N$ implies $h^{-1}_* g_* \in N$. It is clear that $gN \mapsto g_* N$ is a homomorphism whenever it is well-defined. Conversely, for each homomorphism $f: G_N \rightarrow G_N$, we can choose some finite subset $F=\{w_1,\dots,w_n\} \subseteq \mathbb{F}_n$ of size $n$ (if necessary, by multiplying some $w_i$'s by elements from $N$) with the property that $f(\gamma_i N)=w_i N$ such that $h^{-1} g \in N$ implies $h^{-1}_* g_* \in N$.

Now fix a subset $F \subseteq \mathbb{F}_n$ of size $n$. Clearly the set of $N$'s in $\spaceofg_n$ for which the map given by $gN \rightarrow g_*N$ is
\begin{itemize}
\item well-defined is defined by the sentence $\bigwedge_{g,h \in \mathbb{F}_n}\left(g=h \rightarrow g_*=h_*\right)$
\item surjective is defined by the sentence $\bigwedge_{h \in \mathbb{F}_n}\bigvee_{g \in \mathbb{F}_n} g_*=h$
\item injective is defined by the sentence $\bigwedge_{g,h \in \mathbb{F}_n} \left(g_*=h_* \rightarrow g=h\right)$
\end{itemize}
It follows that the set of Hopfian groups in $\spaceofg_n$ is defined by the sentence
\[ \bigwedge_{\substack{F \subseteq \mathbb{F}_n \\ F \text{ has size }n}} \left(\bigwedge_{g,h \in \mathbb{F}_n}\left(g=h \rightarrow g_*=h_*\right)\ \wedge\ \bigwedge_{h \in \mathbb{F}_n}\bigvee_{g \in \mathbb{F}_n} g_*=h\right)\rightarrow \bigwedge_{g,h \in \mathbb{F}_n} \left(g_*=h_* \rightarrow g=h\right)\]
Similarly, the set of co-Hopfian groups in $\spaceofg_n$ is defined by the sentence
\[ \bigwedge_{\substack{F \subseteq \mathbb{F}_n \\ F \text{ has size }n}} \left(\bigwedge_{g,h \in \mathbb{F}_n}\left(g=h \longleftrightarrow g_*=h_*\right)\right) \rightarrow \left(\bigwedge_{h \in \mathbb{F}_n}\bigvee_{g \in \mathbb{F}_n} g_*=h\right)\]
Since each sentence of the form $g=h$ and $g_*=h_*$ defines a $\bdelta^0_1$-set, it is easily seen that the sets of Hopfian and co-Hopfian groups in $\spaceofg_n$ are in $\bpi^0_3$ and $\bpi^0_2$ respectively.

Since finite groups, which are co-Hopfian, can converge to a non-abelian free group, which is not co-Hopfian, the set of co-Hopfian groups in $\spaceofg_n$ is not in $\bpi^0_1$. As mentioned before, the Baumslag-Solitar groups $B(m,n)$, which are not Hopfian, can converge to $\mathbb{F}_2$, which is Hopfian, in $\spaceofg_2$ and hence the set of Hopfian groups in $\spaceofg_n$ is not in $\bsigma^0_1$. We do not know whether or not the set of Hopfian groups (respectively, co-Hopfian groups) is in $\bpi^0_1$ (respectively, in $\bsigma^0_1$.)

\section{Some remarks and open questions}

While we determined the descriptive complexity of some fundamental group theoretic properties, one can further this study by analyzing more technical properties. For example, a non-trivial result of \cite{Shalom00} shows that Kazhdan's property (T) is a $\bsigma^0_1$-property. Since finite groups have the property (T) and free groups do not, this property is clearly not $\bpi^0_1$ and hence, the set of groups in $\spaceofg_n$ with Kazhdan's property (T) is $\bsigma^0_1$-complete.

Another direction in which this study can be taken is the analysis of how the complexity of being $P$ changes if one considers being \textit{just-non} $P$ or residually $P$. Recall that, given a property $P$ of groups, a group is said to be just-non $P$ if it is not $P$ and all its proper quotients are $P$.

Let us remind that a subset of a Polish space $X$ is called analytic if it is the projection of some Borel subset of $X \times Y$ for some Polish space $Y$; and called coanalytic if its complement is analytic. Given a property $P$ which defines a Borel set in $\spaceofg_n$, a quick examination shows that, in general, the properties of being residually $P$ and just-non $P$ respectively define analytic and coanalytic sets. Though, these sets need not be non-Borel. For example, as we have observed earlier, being residually finite is a Borel property. We currently know of no Borel properties $P$ for which being residually $P$ or just-non $P$ define non-Borel sets. However, such a result would seem to be important as it would show that there is no equivalent possible ``Borel definition". (An interesting result that is along the same lines is the result of Wesolek and Williams \cite{WW2017} which shows that, while amenability is a Borel property, the set of elementary amenable groups in $\spaceofg_n$ is coanalytic non-Borel.) We would like to conclude this paper by listing some open questions that we think are of importance.

\textbf{Question.} What are the descriptive complexities of the sets of simple groups, amenable groups, residually finite groups, Hopfian groups and co-Hopfian groups in $\spaceofg_n$? In particular, are the sets of simple groups, amenable groups and co-Hopfian groups $\bpi^0_2$-complete?

\textbf{Question.} Are there group theoretic Borel properties $P$ for which being just-non $P$ or being residually $P$ define non-Borel sets in $\spaceofg_n$?

\appendix
\section{Grigorchuk groups}

In this section we will recall the construction of Grigorchuk \cite{Grigorchuk84}  and its slight modification from \cite{BenliGrigorchuk14}.
Although the original definition is in terms of measure preserving transformations of the unit interval, we will give here a definition in terms of automorphisms of the binary rooted tree.

Let $T_2$ denote the binary rooted tree, where we identify each vertex by a finite binary sequence in  $\{0,1\}^\ast$ (The root of $T_2$ is identified with the empty sequence). $Aut(T_2)$ denotes the automorphism group of $T_2$. Given a vertex $v$ of $T_2$ and an automorphism $f\in Aut(T_2)$, the \textit{section} of $f$ at $v$ is the automorphism $f_v$ defined uniquely by the equation:
$$f(vu)=f(u)f_v(u)\;\;\text{for any vertex}\;\; u.$$
Also, the \textit{root permutation} of $f$ is the permutation $\sigma_f\in S_2$ defined by $\sigma_f(x)=y \iff f(x)=y$.
One observes that the map $\Phi : Aut(T_2) \to Aut(T_2)\wr S_2$ defined by  $\Phi(f)=(f_0,f_1;\sigma_f)$ is an isomorphism.

Let $\tau : 3^\mathbb{N} \to  3^\mathbb{N}$
be the shift i.e., $\tau(\om_0\om_1\om_2\ldots )=\om_1\om_2\ldots$.
For each 
$\om=\om_0\om_1 \ldots \in  3^\mathbb{N}$ we will define a subgroup $G_\om$
of $Aut(\mathcal{T}_2)$.
Each group $G_\om$ is the subgroup generated by the four automorphisms denoted by
$a,b_\om,c_\om,d_\om$
whose actions onto the tree is given (recursively) by the following:

\smallskip

$$a(0v)=1v \;\text{and} \; a(1v)=0v  \;\;\text{for any}\;\; v\in \{0,1\}^*$$

$$\begin{array}{llllll}
 b_\om(0v)=& 0 \beta(\om_0)(v) &   c_\om(0v)=& 0 \zeta(\om_0)(v)
  &  d_\om(0v)= &0 \delta(\om_0)(v) \\

   b_\om(1v)=& 1 b_{\tau \om}(v) &      c_\om(1v)= & 1 c_{\tau \om}(v)
   & d_\om(1v)= & 1 d_{\tau \om}(v), \\

\end{array}
$$
where
$$
\begin{array}{ccc}
 \beta(0)=a & \beta(1)=a & \beta(2)=e \\
  \zeta(0)=a & \zeta(1)=e & \zeta(2)=a \\
   \delta(0)=e & \delta(1)=a & \delta(2)=a \\
\end{array}
$$
and $e$ denotes the identity automorphism.  In the language of sections, this can be written shortly as $(b_\alpha)_0=\beta(\alpha_0)$ and $(b_\alpha)_1=b_{\tau \alpha}$ etc. Note that the isomorphism 
$\Phi$ restricts to  an embedding of $G_\alpha$ into $G_{\tau \alpha}\wr S_2$.

Denoting by $S_\om=(a,b_\om,c_\om,d_\om)$, we obtain a subset $\{(G_\om,S_\om) \mid \om \in 3^\mathbb{N}\}\subseteq \m{G}_4$.
Let $E\subseteq 3^\mathbb{N}$ denote the set of eventually constant sequences. In \cite{Grigorchuk84}, it was observed if two 
sequences $\alpha,\beta \in 3^\mathbb{N}- E$ have a common prefix of length $n\ge 1$, then the marked groups $(G_\alpha,S_\alpha),
(G_\beta,S_\beta)$ are within $2^{n-1}$ in $\m{G}_4$. Thus, replacing the marked groups $(G_\alpha,S_\alpha),\alpha \in E$ by the appropriate limits (denoted again by $(G_\alpha,S_\alpha)$) in 
$\m{G}_4$, one obtains
a closed subset $\m{R}=\{(G_\om,S_\om)\mid \om \in 3^\mathbb{N}\}\subset \m{G}_4 $. Also, this new family also has the property that 
$G_\alpha$ embeds into $G_{\tau \alpha}\wr S_2$ by a map  analogous  to $\Phi$.

Let $C\subseteq
3^\mathbb{N}$ be  the set of recursive sequences and   
$I\subseteq 3^\mathbb{N}$  denote the set of sequences  which contain infinitely many 0's, 1's and 2's. Then we have the following.
 
 \begin{theorem}\cite{Grigorchuk84} \mbox{}
 \begin{itemize}
  \item[1)] $\m{R}\subseteq \m{G}_4$ is homeomorphic to $3^\mathbb{N}$ via the map $\om \mapsto (G_\om,S_\om)$,
  \item[2)] $G_\om$ is a periodic group if and only if $\om \in I$.
  \item[3)] $G_\om$ has exponential  growth for $\om \in E$ and intermediate growth for $\om \notin E$.
  \item[4)] If  $\om \in E$, then $G_\om$ is virtually metabelian.
  \item[5)] $G_{\alpha}$ has decidable word problem if and only if $\alpha \in C$.

 \end{itemize}
 \end{theorem}

In  \cite[Theorem 2]{BenliGrigorchuk14}, it was proven that if 
$\om \in 3^\mathbb{N}$ is a constant sequence, then 
$G_\alpha \cong L\rtimes \mathbb{Z}_2$, where $L=\mathbb{Z}_2\wr \mathbb{Z}$ is the Lamplighter group. In particular, $G_\alpha$ is solvable for constant $\alpha$. Since $G_\alpha$ embeds into $G_{\tau \alpha}\wr S_2$, one deduces that  for every $\om\in E$, $G_\om$ is solvable.

 To obtain a family in $\mathcal{G}_2$ with similar properties, we follow \cite{BenliGrigorchuk14}. For each $\om \in 3^\mathbb{N}$, let  $T_\om=\{d_\om,ab_\om\}\subset G_\alpha$ and $L_\om=\langle T_\om \rangle \le G_\alpha$. 
Also let $\mathcal{L}= \{(L_\om,T_\om) \mid \om \in 3^\mathbb{N}\}\subseteq \m{G}_2$. We have the following.

 \begin{theorem}
\cite{BenliGrigorchuk14}\mbox{}
 \begin{itemize}
  \item[1)] $\m{L}\subseteq \mathcal{G}_2$ is homeomorphic to $3^\mathbb{N}$ via the map $\om \mapsto (L_\om,T_\om)$,
  \item[2)] $L_\om$ is a periodic group if and only if $\om \in I$.
  \item[3)] $L_\om$ has exponential  growth for $\om \in E$ and intermediate  growth otherwise.
  \item[4)]   $L_\om$ is solvable if and only if  $\om\in E $.
  \item[5)] $L_{\alpha}$ has decidable word problem if and only if $\alpha \in C$.
   \end{itemize}
 \end{theorem}

\bibliography{references}{}
\bibliographystyle{amsalpha}

\end{document}